\documentclass[12pt]{amsart}
\usepackage{amsfonts, amssymb, latexsym, graphicx, hyperref, amsmath}
\usepackage[vcentermath]{youngtab}
\usepackage{tikz}
\usepackage{tikz-cd}
\usetikzlibrary{calc, shapes, backgrounds,arrows,positioning,decorations.pathmorphing}

\newtheorem{theorem}{Theorem}%[section]

\newtheorem{lemma}[theorem]{Lemma}

\newtheorem*{claim*}{Claim}
\newtheorem{corollary}[theorem]{Corollary}
\newtheorem{Main Conjecture}[theorem]{Main Conjecture}

\theoremstyle{remark}

\theoremstyle{plain}

\newcommand{\cellsize}{14}
\newlength{\cellsz} 
\newsavebox{\cell}
\sbox{\cell}{\begin{picture}(\cellsize,\cellsize)
\put(0,0){\line(1,0){\cellsize}}
\put(0,0){\line(0,1){\cellsize}}
\put(\cellsize,0){\line(0,1){\cellsize}}
\put(0,\cellsize){\line(1,0){\cellsize}}
\end{picture}}
\newcommand\cellify[1]{\def\thearg{#1}\def\nothing{}%
\ifx\thearg\nothing
\vrule width0pt height\cellsz depth0pt\else
\hbox to 0pt{\usebox{\cell} \hss}\fi%
\vbox to \cellsz{
\vss
\hbox to \cellsz{\hss$#1$\hss}
\vss}}
\newcommand\tableau[1]{\vtop{\let\\\cr
\baselineskip -16000pt \lineskiplimit 16000pt \lineskip 0pt
\ialign{&\cellify{##}\cr#1\crcr}}}

\newcommand{\excise}[1]{}%{$\star$\textsc{#1}$\star$}

\begin{document}
\pagestyle{plain}
\title{Super-Catalan numbers of the Third and Fourth Kind}
\author{Irina Gheorghiciuc and Gidon Orelowitz}
\address{Dept.~of Mathematics, Carnegie Mellon  University, Pittsburgh, PA 15213, USA} 
\email{gheorghi@andrew.cmu.edu}
\address{Dept.~of Mathematics, U.~Illinois at Urbana-Champaign, Urbana, IL 61801, USA} 
\email{gidono2@illinois.edu}
\date{\today}
\begin{abstract}
The \emph{Super-Catalan numbers} are a generalization of the Catalan numbers defined as $T(m,n) =  \frac{(2m)!(2n)!}{2m!n!(m+n)!}$.  It is an open problem to find a combinatorial interpretation for $T(m,n)$.  We resolve this for $m=3,4$ using a common form; no such solution exists for $m=5$.
\end{abstract}
\maketitle
\section{Introduction}
\subsection{Main Results}

The \emph{Super-Catalan numbers}, first described by Catalan in 1874, are
\begin{equation}
T(m,n) := \frac{\binom{2m}{m}\binom{2n}{n}}{2\binom{m+n}{m}} = \frac{1}{2}\cdot\frac{(2m)!(2n)!}{m!n!(m+n)!}.
\end{equation}

In \cite{SuperBallot}, Gessel shows that $T(m,n)$ is a positive integer for all $(m,n)\in \mathbb{Z}_{\geq 0} \times \mathbb{Z}_{>0}$.  It is natural to ask if a combinatorial interpretation of $T(m,n)$ exists.  $T(0,n)=\binom{2n-1}{n}$ by observation, and $T(1,n)=C_n$, the Catalan numbers.  This sequence has numerous combinatorial interpretations;  see, e.g., Stanley's \cite{Stanley}.  The one that this paper will use is that $C_n$ is the number of Dyck paths of length $2n$.

A \emph{Dyck path} is a sequence $\pi = (\pi(1),\pi(2),\dots, \pi(2n))$ such that \[\pi(k)=\pm 1,\  \sum_{i=1}^{2n} \pi(i) = 0, \text{ and } \sum_{i=1}^k \pi(i) \geq 0 \text{ for all } 1\leq k \leq 2n.\]  The \emph{length} of $\pi$, denoted $|\pi|$, is the number of elements in the sequence.  The \emph{total length} of a tuple of Dyck paths $(\pi^{(1)},\dots, \pi^{(n)})$ is $\sum_{i=1}^n |\pi^{(i)}|$.  The  \emph{height} of a Dyck path $\pi$ is \[h(\pi): = \max_k\sum_{i=1}^k \pi(i),\] and say $h(\epsilon) = 0$, where $\epsilon$ is the unique Dyck path of length 0.

Let 
\begin{equation}
\label{def:G}
G_n(m,k):= |\{(\pi^{(1)},\dots, \pi^{(m)}): \sum_{i=1}^m |\pi^{(i)}| = 2n,|h(\pi^{(i)})-h(\pi^{(j)})|\leq k\}|.
\end{equation}

\begin{theorem}
\label{m3Theorem}
\begin{equation}
\label{m=3}
T(3,n) = G_n(3,1) + 2G_{n-1}(3,0)
\end{equation}
\end{theorem}

Theorem \ref{m3Theorem} is used to prove:

\begin{theorem}
\label{m4Theorem}
\begin{equation}
\label{m=4}
T(4,n) = G_n(4,1) + 10G_{n-1}(4,0)+ 4G_{n-2}(5,0)
\end{equation}
\end{theorem}

Since $G_n(m,k)$ is defined combinatorially in (\ref{def:G}), (\ref{m=3}) and (\ref{m=4}) immediately imply a combinatorial interpretation for $T(3,n)$ and $T(4,n)$.

Vacuously, \[G_n(1,k) = C_n \text{ for all } k,n\in \mathbb{Z}_{\geq 0}, \] so in particular
\begin{equation}
\label{m=1}
T(1,n) = G_n(1,1).
\end{equation}
In \cite{GesselT2n}, Gessel and Xin show that 
\begin{equation}
\label{m=2}
T(2,n)= G_n(2,1).
\end{equation}
Each of (\ref{m=1}), (\ref{m=2}), (\ref{m=3}), and (\ref{m=4}) are of the form
\begin{equation}
\label{conjecture}
T(m,n) = G_n(m,1) + \sum_{k=1}^{m-2} a_{k,m}G_{n-k}(m+k-1,0)
\end{equation}
for some $a_{k,m}\in \mathbb{Z}_{> 0}$.  However, there are no choices for $a_{1,5}$, $a_{2,5}$, and $a_{3,5}$ in $\mathbb{R}$ that simultaneously satisfy (\ref{conjecture}) when $m=5$ and $n=1,2,3$.  Therefore, (\ref{conjecture}) is not true for $m=5$. 

\subsection{Comparison to Literature}

Theorem 1 in \cite{T2n} shows that \[T(m,n)=P(m,n)-N(m,n),\] where $P(m,n)$ and $N(m,n)$ are the number of $m$-positive and $m$-negative Dyck paths of length $2m+2n-2$ respectively.  A Dyck path $\pi$ is \emph{$m$-positive} (respectively \emph{$m$-negative}) if
\begin{equation}
\sum_{i=1}^{2m-1} \pi(i)\equiv 1 \text{ mod 4 (3 mod 4 respectively).}
\end{equation}

A second identity, proved in \cite{Georgiadis}, says that \[2T(m,n)= (-1)^m\sum_{\pi\in \mathcal{P}_{m+n}} (-1)^{h_{2n}(\pi)}.\]
Here, $\mathcal{P}_{m+n}$ is the set of Dyck paths of length $2m+2n$ and 
\begin{equation}
h_{n}(\pi) = |\{i : \pi(i)=1,i>n\}|.
\end{equation}

The most relevant identity of the Super-Catalan numbers for this paper is that
\begin{equation}
\label{baserecursion}
T(m+1,n)=4T(m,n)-T(m,n+1),
\end{equation}
which is attributed to Rubenstein in \cite{SuperBallot}.

\section{The Path-Height Function}
In \cite{T2n}, Allen and Gheorghiciuc give a number of useful definitions for a non-empty Dyck path $\pi$.  In their paper, the \emph{$R$-point} is defined as 
\begin{equation}
R(\pi) := \max\{k:h(\pi) = \sum_{i=1}^k \pi(i)\}
\end{equation}
and the \emph{$X$-point} is 
\begin{equation}
X(\pi) := \max\{k\leq R(\pi):\sum_{i=1}^k \pi(i)=1\}.
\end{equation} 
Additionally, they define 
\begin{equation}
h_-(\pi):= \max_{1\leq k \leq X(\pi)}(\sum_{i=1}^{k}\pi(i)).
\end{equation}

\begin{lemma}
\label{trivial}
Let $\pi$ be a non-empty Dyck path.
\begin{enumerate}
\item $\pi(1) = 1$ and $\pi(|\pi|) = -1$.
\item $\pi(R(\pi)) = 1$ and $\pi(R(\pi)+1) = -1$.
\item $1\leq h_-(\pi)\leq h(\pi)$.
\item If $h(\pi) > 1$, then $\pi(X(\pi)+1) = 1$, and if $h(\pi) > 2$, then $\pi(X(\pi)+2) = 1$.
\end{enumerate}
\end{lemma}
\begin{proof}
These follow from the definitions of Dyck path, $R$-point, and $X$-point.
\end{proof}

Define the \emph{path-height function} to be 
\begin{equation}
P_n(a_1,\dots,a_m) := |\{(\pi^{(1)},\dots, \pi^{(m)}):h(\pi^{(i)})=a_i, \sum_{i=1}^m |\pi^{(i)}| = 2n\}|
\end{equation}
when $a_1,\dots, a_m\in \mathbb{Z}_{\geq 0}$, and $P_n(a_1,\dots,a_m)=0$ otherwise.  Notice that $P_n(a_1,\dots,a_m)$ is symmetric in $a_1,\dots, a_m$.  The function also has the recursive formulation
\begin{equation}
\label{Pn Calculation}
P_n(a_1,\dots, a_m) = \sum_{n_1+\dots n_m =n }\left( \prod_{i=1}^m P_{n_i}(a_i)\right),
\end{equation}
which provides a fast way to compute $P_n(a_1,\dots, a_m)$ electronically.  Since the empty path is the unique Dyck path of height 0, and has length 0,
\begin{equation}
\label{zero}
P_n(a_1,\dots, a_m) = P_n(a_1,\dots,a_m,0).
\end{equation}

$G_n(m,k)$ can be written as the sum of path-height functions as follows:
\begin{equation}
G_n(m,k) = \sum_{|a_i-a_j|\leq k} P_n(a_1,\dots, a_m).
\end{equation}

To proceed further, it is necessary to prove two lemmas and several corollaries related to the path-height function.

\begin{lemma}
\label{splitter}
For $a_m \geq 1$, 
\begin{equation}
\label{eqn:splitter}
P_n(a_1,\dots, a_{m-1},a_m)=  \sum_{y=1}^{a_m} P_n(a_1,\dots, a_{m-1},a_m-1,y).
\end{equation}
\end{lemma}

\begin{proof}
If $a_m=1$, then the lemma follows from Equation \ref{zero} and the symmetry of $P_n$.

Now assume $a_m>1$.  For any $h_-,h,n\in\mathbb{Z}_{>0}$ with $1\leq h_- \leq h$, $h\geq 2$, \cite[Theorem~2]{T2n} provides a bijection between the set $\{\pi:h_-(\pi) = h_-,h(\pi) = h, |\pi| = 2n\}$ and the set $\{(\pi^{(1)},\pi^{(2)}):h(\pi^{(1)})=h_-, h(\pi^{(2)}) = h-1, |\pi^{(1)}|+|\pi^{(2)}|=2n\}$.  As a result, there is a bijection between
\begin{equation*}
\bigcup_{h_- = 1}^h \{\pi:h_-(\pi) = h_-,h(\pi) = h, |\pi| = 2n\} = \{\pi:h(\pi) = h, |\pi| = 2n\}
\end{equation*}
and
\begin{align*}
\bigcup_{h_- = 1}^h \{(\pi^{(1)},\pi^{(2)}):h(\pi^{(1)})=h_-, h(\pi^{(2)}) = h-1, |\pi^{(1)}|+|\pi^{(2)}|=2n\}\\
= \{(\pi^{(1)},\pi^{(2)}):1\leq h(\pi^{(1)})\leq h, h(\pi^{(2)}) = h-1, |\pi^{(1)}|+|\pi^{(2)}|=2n\}.
\end{align*}

And so, as a result,
\begin{align*}
P_n(h) =&\  |\{\pi:h(\pi) = h, |\pi| = 2n\}|\\
=&\  |\{(\pi^{(1)},\pi^{(2)}):1\leq h(\pi^{(1)})\leq h, h(\pi^{(2)}) = h-1, |\pi^{(1)}|+|\pi^{(2)}|=2n\}| = \sum_{y=1}^h P_n(y,h-1).
\end{align*}

Combining this with Equation \ref{Pn Calculation} and the fact that the path-height function is symmetric completes the proof.
\end{proof}

\begin{corollary}
\label{cor:splitter2}
\begin{equation}
\label{splitter2}
\begin{split}
P_n(a_1,\dots, a_{m-1},a_m-1,a_m+1)=&\   P_n(a_1,\dots, a_{m-1},a_m,a_m)\\
&\ +P_n(a_1,\dots, a_{m-1},a_m-1,a_m,a_m+1)
\end{split}
\end{equation}
\end{corollary}

\begin{proof}
\begin{align*}
P_n(a_1,\dots,& a_{m-1},a_m-1,a_m+1)\\
&\ =\sum_{y=1}^{a_m+1} P_n(a_1,\dots, a_{m-1},a_m-1,a_m,y) \\
&\ =P_n(a_1,\dots, a_{m-1},a_m-1,a_m,a_m+1)+ \sum_{y=1}^{a_m} P_n(a_1,\dots, a_{m-1},a_m-1,a_m,y) \\
&\ =P_n(a_1,\dots, a_{m-1},a_m-1,a_m,a_m+1)+ P_n(a_1,\dots, a_{m-1},a_m,a_m) \qedhere
\end{align*}\end{proof}

\begin{lemma}
\label{shorten}
For $n, a_m\in \mathbb{Z}_{>0}$,
\begin{align*}
P_n(a_1,\dots, a_{m-1},a_m)
=&\ P_{n-1}(a_1,\dots, a_{m-1},a_m-1) -P_{n-1}(a_1,\dots, a_{m-1},a_m,a_m-1)\\
&\ +2P_{n-1}(a_1,\dots, a_{m-1},a_m) -P_{n-1}(a_1,\dots, a_{m-1},a_m,a_m)\\
&\ +P_{n-1}(a_1,\dots, a_{m-1},a_m+1) -P_{n-1}(a_1,\dots, a_{m-1},a_m,a_m+1).
\end{align*}
\end{lemma}

\begin{proof}
This proof is broken up into three cases depending on whether $a_m=1$, $a_m=2$, or $a_m\geq 3$.

${\sf Case\ 1}:(a_m =1)$  Consider an arbitrary tuple of Dyck paths $(\pi^{(1)},\dots, \pi^{(m)})$ with total length $2n$ and respective heights $a_1,\dots, a_m$ such that $a_m = 1$.  Since $h(\pi^{(m)})=1$, $\pi^{(m)}(1)=1$ and $\pi^{(m)}(2)=-1$, so removing the first two steps of $\pi^{(m)}$ results in a new Dyck path, denoted $\pi^{(m)'}$, with the properties that $|\pi^{(m)'}|=|\pi^{(m)}|-2$ and $h(\pi^{(m)'})=0$ or 1.  This process is an injection, and it is reversible, as any path of height 0 or 1 can have $(1, -1)$ prepended to it to result in a unique path of length 1.  Since the map $\pi^{(m)} \mapsto \pi^{(m)'}$ is a bijection, so is $(\pi^{(1)},\dots, \pi^{(m-1)}, \pi^{(m)})\mapsto (\pi^{(1)},\dots, \pi^{(m-1)}, \pi^{(m)'})$, and so there is a bijection between 
\begin{align*}
\{(\pi^{(1)},\dots, \pi^{(m)}):\sum_{i=1}^m |\pi^{(i)}| = 2n,h(\pi^{(i)})=a_i, h(\pi^{(m)}) = 1\}
\end{align*}
and
\begin{align*}
\{(\pi^{(1)},\dots, \pi^{(m-1)}, \pi^{(m)'}):\sum_{i=1}^{m-1} |\pi^{(i)}|  + |\pi^{(m)'}|= 2n-2,h(\pi^{(i)})=a_i, 0\leq h(\pi^{(m)'})\leq 1\}.
\end{align*}
By definition, the size of the first set is $P_n(a_1,\dots, a_{m-1},1)$, and the size of the second set is $P_{n-1}(a_1,\dots, a_{m-1},1) + P_{n-1}(a_1,\dots, a_{m-1},0)$.  As a result,
\begin{align*}
P_n(a_1,\dots, a_{m-1},1)=&\ P_{n-1}(a_1,\dots, a_{m-1},1) + P_{n-1}(a_1,\dots, a_{m-1})\\
=&\ P_{n-1}(a_1,\dots, a_{m-1},1) + P_{n-1}(a_1,\dots, a_{m-1},0)\\
 &\ + P_{n-1}(a_1,\dots, a_{m-1},1) - P_{n-1}(a_1,\dots, a_{m-1},1)\\
 &\ +P_{n-1}(a_1,\dots, a_{m-1},2) - P_{n-1}(a_1,\dots, a_{m-1},2).
\end{align*}
Applying Lemma \ref{splitter} to the last term results in
\begin{align*}
P_n(a_1,\dots, a_{m-1},1)=&\ P_{n-1}(a_1,\dots, a_{m-1},1) + P_{n-1}(a_1,\dots, a_{m-1},0)\\
 &\ + P_{n-1}(a_1,\dots, a_{m-1},1) - P_{n-1}(a_1,\dots, a_{m-1},0,1)\\
 &\ +P_{n-1}(a_1,\dots, a_{m-1},2)\\
 &\ - (P_{n-1}(a_1,\dots, a_{m-1},1,1) + P_{n-1}(a_1,\dots, a_{m-1},1,2))\\ 
=&\ 2P_{n-1}(a_1,\dots, a_{m-1},1) + P_{n-1}(a_1,\dots, a_{m-1},0)+ P_{n-1}(a_1,\dots, a_{m-1},2)\\
  &\ - P_{n-1}(a_1,\dots, a_{m-1},1,0)\\
  &\ - P_{n-1}(a_1,\dots, a_{m-1},1,1) - P_{n-1}(a_1,\dots, a_{m-1},1,2),
\end{align*}
  
  which proves this case.

${\sf Case\ 2}:(a_m =2)$ Consider an arbitrary tuple of Dyck paths $(\pi^{(1)},\dots, \pi^{(m)})$ with total length $2n$ and respective heights $a_1,\dots, a_m$ such that $a_m = 2$.  Since $h(\pi^{(m)})=2$,  $(\pi^{(m)}(1),\pi^{(m)}(2),\pi^{(m)}(3)) = (1,-1,1)$ or $(1,1,-1)$.  If $(\pi^{(m)}(1),\pi^{(m)}(2),\pi^{(m)}(3)) = (1,-1,1)$, then as in Case 1, removing the first two steps of $\pi^{(m)}$ is a bijection between these Dyck paths and Dyck paths with length $|\pi^{(m)}|-2$ and height 2.  If instead $(\pi^{(m)}(1),\pi^{(m)}(2),\pi^{(m)}(3)$) is $(1,1,-1)$, then the bijection is $\pi^{(m)} \mapsto (1,\pi^{(m)}(4),\dots,\pi^{(m)}(|\pi_m|))$.  In this case, the image of this bijection is Dyck paths of length $|\pi^{(m)}|-2$ and height 1 or 2.  Clearly, if $\pi^{(m)}$ is the last element of an $m$-tuple, the mapping is still bijective, so
\begin{align*}
|\{(&\pi^{(1)},\dots, \pi^{(m-1)}, \pi^{(m)}):\sum_{i=1}^m |\pi^{(i)}| = 2n,h(\pi^{(i)})=a_i, h(\pi^{(m)}) = 2\}|\\
=&\ |\{(\pi^{(1)},\dots, \pi^{(m-1)},\pi^{(m)'}):\sum_{i=1}^{m-1} |\pi^{(i)}| + |\pi^{(m)'}|= 2n-2,h(\pi^{(i)})=a_i, h(\pi^{(m)'}) = 2\}|\\
&\ +|\{(\pi^{(i)},\dots, \pi^{(m-1)},\pi^{(m)'}):\sum_{i=1}^{m-1} |\pi^{(i)}| + |\pi^{(m)'}|= 2n-2,h(\pi^{(i)})=a_i, 1\leq h(\pi^{(m)'}) \leq 2\}|
\end{align*}
which by definition means that
\begin{align*}
P_n(a_1,\dots, a_{m-1},2)=&\ 2P_{n-1}(a_1,\dots, a_{m-1},2) + P_{n-1}(a_1,\dots, a_{m-1},1)\\
=&\ 2P_{n-1}(a_1,\dots, a_{m-1},2) + P_{n-1}(a_1,\dots, a_{m-1},1)\\
&\ + P_{n-1}(a_1,\dots, a_{m-1},3) - P_{n-1}(a_1,\dots, a_{m-1},3).
\end{align*}
 Applying Lemma \ref{splitter} to the last term of this expression results in:
\begin{align*}
P_n(a_1,\dots, a_{m-1},2) =&\ 2P_{n-1}(a_1,\dots, a_{m-1},2) + P_{n-1}(a_1,\dots, a_{m-1},1)\\
&\ + P_{n-1}(a_1,\dots, a_{m-1},3)- P_{n-1}(a_1,\dots, a_{m-1},2,1)\\
&\  - P_{n-1}(a_1,\dots, a_{m-1},2,2) - P_{n-1}(a_1,\dots, a_{m-1},2,3)
\end{align*}
 which completes this case.

${\sf Case\ 3}:(a_m \geq 3)$ Let $\pi^{(m)}$ be an arbitrary Dyck path such that $h(\pi^{(m)})= a_m\geq 3$.  There are three possiblities for $(\pi^{(m)}(1),\pi^{(m)}(2),\pi^{(m)}(3))$:  either $(1,-1,1)$, $(1,1,-1)$ or $(1,1,1)$.

For the first option, removing the first two steps results in a Dyck path that has height $h(\pi^{(m)})$ and has length $|\pi^{(m)}| -2$.  Likewise, for the second option, removing the second and third steps also results a Dyck path that has the same height $h(\pi^{(m)})$ and has length $|\pi^{(m)}| -2$.  Both of these processes are reversible, and so just as in the previous two cases, the number of $m$-tuples of paths $(\pi^{(1)}, \dots, \pi^{(m)})$ such that for each $1\leq i \leq m$, $h(\pi^{(i)})= a_i$ and $\pi^{(m)}$ starts with either $(1,1,-1)$ or $(1,-1,1)$ is equal to $2P_{n-1}(a_1,\dots, a_{m-1},a_m)$.

Now consider an arbitrary $m$-tuple of paths $(\pi^{(1)}, \dots, \pi^{(m)})$ such that for each $1\leq i \leq m$, $h(\pi^{(i)})= a_i$ and $\pi^{(m)}(1)=\pi^{(m)}(2)=\pi^{(m)}(3) = 1$.  There are two possibilities here to consider: either $X(\pi^{(m)}) = 1$, or $X(\pi^{(m)}) \geq 5$.  In the former of these two cases, \cite[Theorem~3]{T2n} has a bijection between these Dyck paths and Dyck paths of height $h(\pi^{(m)})-1$ and length $|\pi^{(m)}|-2$, so the number of $m$-tuples of Dyck paths in this case is $P_{n-1}(a_1,\dots, a_{m-1},a_m-1)$.

In the final case, \cite[Theorem~3]{T2n} provides a bijection between Dyck paths $\pi$ such that $\pi(1)=\pi(2)=\pi(3) = 1$ and $X(\pi)\geq 5$ to paths $\pi'$ such that $|\pi'| = |\pi|-2$, $h(\pi')=h(\pi)+1$, and $h_-(\pi')<h(\pi')-2$. 

As a result, the number of $m$-tuples in this case equals $P_{n-1}(a_1,\dots, a_{m-1},a_m+1)$ minus the the number of $m$-tuples of Dyck paths counted by $P_{n-1}(a_1,\dots, a_{m-1},a_m+1)$ such that $h(\pi^{(m)})-2\leq h_-(\pi^{(m)})\leq h(\pi^{(m)})$.  According to the bijection mentioned in the proof for Lemma \ref{splitter}, the number of Dyck paths that are being subtracted is equal to $\sum_{y=a_m-1}^{a_m+1}P_{n-1}(a_1,\dots, a_{m-1},a_m,y)$, and so in total
\begin{align*}
P_n(a_1,\dots, a_{m-1},a_m)=&\ 2P_{n-1}(a_1,\dots, a_{m-1},a_m)\\
  &\ + P_{n-1}(a_1,\dots, a_{m-1},a_m-1)\\
  &\ + P_{n-1}(a_1,\dots, a_{m-1},a_m+1)\\
  &\ - \sum_{y=a_m-1}^{a_m+1}P_{n-1}(a_1,\dots, a_{m-1},a_m,y).  \qedhere
\end{align*}
\end{proof}

Note that because the Path-Height function is symmetric, either of the above lemmas or the related corollaries can be applied to any value in the argument, rather than just the final value.  For example, $P_n(3,4) = \sum_{z=1}^4 P_n(3,3,z) = \sum_{y=1}^3 P_n(2,y,4)$.

\section{The Grouped Path-Height Function}

Sometimes when discussing tuples of Dyck paths, the heights of the Dyck paths are less important than the relative values of the heights.  To that end, for $a_0,\dots a_k\in \mathbb{Z}_{\geq 0}$, define the the following intermediary function:
\begin{equation}
Q_n^{(x)}(a_0,\dots, a_k) := P_n(x,\dots, x, x+1,\dots, x+1,\dots, x+k, \dots, x+k)
\end{equation}
where the first $a_0$ inputs on the right-hand side are $x$, the next $a_1$ inputs are $x+1$, and so on.  If $a_i\not\in \mathbb{Z}_{\geq 0}$ for some $0\leq i\leq k$, say the function is zero.  From here, define the \emph{grouped path-height function} to be
\begin{equation}
Q_n(a_0,\dots, a_k) = \sum_{x=0}^\infty Q_n^{(x)}(a_0,\dots, a_k).
\end{equation}
For example, $Q_n(1,0,2) = \sum_{x=0}^\infty P_n(x,x+2,x+2)$.  Many of the results that for $P_n$ can be extended to $Q_n$:
\begin{lemma}
\label{QBasic}
\begin{enumerate}
\item
\begin{equation*}
Q_n(a_0,\dots, a_k,0) = Q_n(a_0,\dots, a_k)
\end{equation*}
\item
\begin{equation*}
Q_n(0,a_0,\dots, a_k) = \sum_{x=1}^\infty Q_n^{(x)}(a_0,\dots, a_k) = Q_n(a_0,\dots, a_k) - Q_n^{(0)}(a_0,\dots, a_k)
\end{equation*}
\item
\begin{equation*}
Q^{(0)}_n(a_0,a_1,\dots) = Q^{(0)}_n(0,a_1,\dots) = Q^{(1)}_n(a_1,\dots)
\end{equation*}
\end{enumerate}
\end{lemma}
\begin{proof}
These follow from the definition of the grouped path-height function.
\end{proof}

\begin{lemma}
\label{QSplitter2}
For $i>0$ and $a_{i-1},a_{i+1}>0$,
\begin{equation}
Q_n(a_0,\dots) = Q_n(a_0,\dots,a_{i-1}-1, a_i+2,a_{i+1}-1,\dots) + Q_n(a_0,\dots,a_{i-1}, a_i+1,a_{i+1},\dots).
\end{equation}
\end{lemma}
\begin{proof}
This follows from Corollary \ref{cor:splitter2}.
\end{proof}

\begin{lemma}
\label{QShorten}
Fix $i,x\geq  0$ and let $a_i>0$.
\begin{enumerate}
\item
If $i>0$,
\begin{align*}
Q^{(x)}_n(a_0,\dots) =&\  2Q^{(x)}_{n-1}(a_0,\dots)\\
&\ + Q^{(x)}_{n-1}(\dots, a_{i-1}+1, a_i -1, \dots) + Q^{(x)}_{n-1}(\dots, a_i-1, a_{i+1} +1, \dots)\\
&\ - Q^{(x)}_{n-1}(\dots, a_{i-1}+1, \dots) -  Q^{(x)}_{n-1}(\dots,a_i + 1,\dots)- Q^{(x)}_{n-1}(\dots, a_{i+1} +1, \dots)
\end{align*}
\item
If $i=0$ and $x>0$,
\begin{align*}
Q^{(x)}_n(a_0,\dots) =&\  2Q^{(x)}_{n-1}(a_0,\dots)\\
&\ + Q^{(x-1)}_{n-1}(1, a_0 -1, \dots) + Q^{(x)}_{n-1}( a_0-1, a_{1} +1, \dots)\\
&\ - Q^{(x-1)}_{n-1}(1, a_0, \dots) -  Q^{(x)}_{n-1}(a_0 + 1,\dots)- Q^{(x)}_{n-1}(a_0, a_{1} +1, \dots)
\end{align*}
\end{enumerate}
\end{lemma}
\begin{proof}
This follows from Lemma \ref{shorten}.
\end{proof}

\begin{corollary}
\label{QShortencor1}
If $i,a_i>0$, 
\begin{align*}
Q_n(a_0,\dots) =&\  2Q_{n-1}(a_0,\dots)\\
&\ + Q_{n-1}(\dots, a_{i-1}+1, a_i -1, \dots) + Q_{n-1}(\dots, a_i-1, a_{i+1} +1, \dots)\\
&\ - Q_{n-1}(\dots, a_{i-1}+1, \dots) -  Q_{n-1}(\dots,a_i + 1,\dots)- Q_{n-1}(\dots, a_{i+1} +1, \dots)
\end{align*}
\end{corollary}
\begin{proof}
This follows from the definition of $Q_n$ and Lemma \ref{QShorten}
\end{proof}

\begin{corollary}
\label{QShortencor2}
If $a>0$, 
\begin{align*}
Q_n(a) =&\   2Q_{n-1}(a)+ Q_{n-1}(1,a-1) + Q_{n-1}(a-1,1)\\
&\ - Q_{n-1}(1,a)-  Q_{n-1}(a+1)- Q_{n-1}(a,1)
\end{align*}
\end{corollary}
\begin{proof}
$Q_n(a) = Q^{(0)}_n(a) + Q_n(0,a) = Q_n(0,a,0)$, so applying Corollary \ref{QShortencor1} results in
\begin{align*}
Q_n(a)=&\ Q_n(0,a,0)\\
=&\ 2Q_{n-1}(0,a,0)+Q_{n-1}(1,a-1,0)+Q_{n-1}(0,a-1,1)\\
&\ -Q_{n-1}(1,a,0)-Q_{n-1}(0,a+1,0)-Q_{n-1}(0,a,1)\\
=&\ 2Q_{n-1}(a)+Q_{n-1}(1,a-1)+Q_{n-1}(0,a-1,1)\\
&\ -Q_{n-1}(1,a)-Q_{n-1}(a+1)-Q_{n-1}(0,a,1)\\
=&\ 2Q_{n-1}(a)+Q_{n-1}(1,a-1)+(Q_{n-1}(a-1,1)-Q^{(0)}_{n-1}(a-1,1))\\
&\ -Q_{n-1}(1,a)-Q_{n-1}(a+1)-(Q_{n-1}(a,1)-Q^{(0)}_{n-1}(a,1))\\
=&\ 2Q_{n-1}(a)+Q_{n-1}(1,a-1)+Q_{n-1}(a-1,1)-Q^{(0)}_{n-1}(0,1)\\
&\ -Q_{n-1}(1,a)-Q_{n-1}(a+1)-Q_{n-1}(a,1)+Q^{(0)}_{n-1}(0,1)\\
=&\ 2Q_{n-1}(a)+ Q_{n-1}(1,a-1) + Q_{n-1}(a-1,1)\\
&\ - Q_{n-1}(1,a) -  Q_{n-1}(a+1)- Q_{n-1}(a,1) \qedhere
\end{align*}
\end{proof}

\begin{lemma}
\label{QShortenab}
If $a,b>0$, 
\begin{align*}
Q_n(a,b) =&\  2Q_{n-1}(a,b)+ Q_{n-1}(a-1,b+1) + Q_{n-1}(a+1,b-1)\\
&\ - Q_{n-1}(a+1,b)- Q_{n-1}(a,b+1)
\end{align*}
\end{lemma}
\begin{proof}
Applying Corollary \ref{QShortencor1} to $b$ results in
\begin{align*}
Q_n(a,b)=&\ 2Q_{n-1}(a,b)+Q_{n-1}(a+1,b-1)+Q_{n-1}(a,b-1,1)\\
&\ -Q_{n-1}(a+1,b)-Q_{n-1}(a,b+1)-Q_{n-1}(a,b,1)\\
=&\ 2Q_{n-1}(a,b)+Q_{n-1}(a+1,b-1)+(Q_{n-1}(a,b-1,1)-Q_{n-1}(a,b,1))\\
&\ -Q_{n-1}(a+1,b)-Q_{n-1}(a,b+1)\\
=&\ 2Q_{n-1}(a,b)+Q_{n-1}(a+1,b-1)+(Q_{n-1}(a-1,b+1))\\
&\ -Q_{n-1}(a+1,b)-Q_{n-1}(a,b+1)
\end{align*}
Where the last step is by Lemma \ref{QSplitter2}.
\end{proof}
This result can be reformulated in a different way.
\begin{lemma}
\label{QShortenReform}
If $(a_k)$ is cyclic with period $m$, then
\begin{align*}
\sum_{k=1}^m a_k Q_n(k,m-k)=&\  \sum_{k=1}^m (2a_{k} + a_{k-1} + a_{k+1})Q_{n-1}(k,m-k)- (a_{k} + a_{k-1})Q_{n-1}(k,m+1-k)\\
&\ -a_mQ_{n-1}(m+1)
\end{align*}
\end{lemma}
\begin{proof}
Applying Corollary \ref{QShortencor1} and Lemma \ref{QShortenab} to the appropriate terms in the summation on the left side results in
\begin{align*}
\sum_{k=1}^m a_k Q_n(k,m-k) =&\  a_mQ_n(m) + \sum_{k=1}^{m-1} a_k Q_n(k,m-k)\\
=&\ 2a_mQ_{n-1}(m)+a_mQ_{n-1}(1,m-1)+a_mQ_{n-1}(m-1,1)\\
&\ - a_mQ_{n-1}(1,m)- a_mQ_{n-1}(m+1) - a_mQ_{n-1}(m,1)\\
&\ + \sum_{k=1}^{m-1} 2a_k Q_{n-1}(k,m-k)\\
&\ + \sum_{k=1}^{m-1} a_k Q_{n-1}(k+1,m-k-1)+ \sum_{k=1}^{m-1} a_k Q_{n-1}(k-1,m-k+1)\\
&\ - \sum_{k=1}^{m-1} a_k Q_{n-1}(k+1,m-k)- \sum_{k=1}^{m-1} a_k Q_{n-1}(k,m+1-k)
\end{align*}
Re-indexing the summations so that they have the same summands produces
\begin{align*}
\sum_{k=1}^m a_k Q_n(k,m-k)=&\ 2a_mQ_{n-1}(m)+a_mQ_{n-1}(1,m-1)+a_mQ_{n-1}(m-1,1)
\end{align*}
\begin{align*}
&\ - a_mQ_{n-1}(1,m)- a_mQ_{n-1}(m+1) - a_mQ_{n-1}(m,1)\\
&\ + \sum_{k=1}^{m-1} 2a_k Q_{n-1}(k,m-k)\\
&\ + \sum_{k=2}^{m} a_{k-1} Q_{n-1}(k,m-k)+ \sum_{k=0}^{m-2} a_{k+1} Q_{n-1}(k,m-k)\\
&\ - \sum_{k=2}^{m} a_{k-1} Q_{n-1}(k,m+1-k)- \sum_{k=1}^{m-1} a_k Q_{n-1}(k,m+1-k)
\end{align*}
Grouping together like terms results in
\begin{align*}
\sum_{k=1}^m a_k Q_n(k,m-k)=&\ - a_mQ_n(m+1)+ \sum_{k=1}^{m} 2a_k Q_n(k,m-k)\\
&\ + \sum_{k=1}^{m} a_{k-1} Q_{n-1}(k,m-k)+ \sum_{k=1}^{m} a_{k+1} Q_{n-1}(k,m-k)\\
&\ - \sum_{k=1}^{m} a_{k-1} Q_{n-1}(k,m+1-k)- \sum_{k=1}^{m} a_k Q_{n-1}(k,m+1-k)
\end{align*}
And so
\begin{align*}
\sum_{k=1}^m a_k Q_n(k,m-k)=&\ \sum_{k=1}^m (2a_{k} + a_{k-1} + a_{k+1})Q_{n-1}(k,m-k)\\
&\ - \sum_{k=1}^m (a_{k} + a_{k-1})Q_{n-1}(k,m+1-k)\\
&\ -a_mQ_{n-1}(m+1)\qedhere
\end{align*}
\end{proof}

\begin{corollary}
\label{QAlternator}
For $m$ even, 
	\begin{equation}
	\sum_{k=1}^{m} (-1)^{k+1} Q_n(k,m-k) = G_{n-1}(m+1,0)
	\end{equation}
\end{corollary}
\begin{proof}
Clearly, $((-1)^{k+1})_{k\in \mathbb{Z}}$ is cyclic with period $m$ for $m$ even.  Therefore, applying Lemma \ref{QShortenReform} to the left hand side results in
\begin{align*}
\sum_{k=1}^{m} (-1)^{k+1} Q_n(k,m-k)=&\ \sum_{k=1}^{m} (2(-1)^{k+1}+(-1)^{k}+(-1)^{k+2}) Q_{n-1}(k,m-k)\\
&\ -\sum_{k=1}^{m} ((-1)^{k+1}+(-1)^{k}) Q_{n-1}(k,m+1-k)\\
&\ -(-1)^{m+1} Q_{n-1}(m+1)\\
=&\ Q_{n-1}(m+1)\\
=&\  G_{n-1}(m+1,0)
\end{align*}
which is the desired equality
\end{proof}
$G_n(m,k)$ can also be expressed in terms of this new function.
\begin{lemma}
\label{GtoQ}
\begin{equation}
G_n(m,k)= \sum_{\substack{m_0>0\\ \sum m_i= m }}\binom{m}{m_0,\dots, m_k} Q_n(m_0,\dots, m_k)
\end{equation}
\end{lemma}
\begin{proof}
$G_n(m,k)= \sum_{|a_i-a_j|\leq k} P_n(a_1,\dots, a_m) = \sum_{x = 0}^\infty \sum_{\substack{a_i\leq x+k\\ \min(a_i)=x}} P_n(a_1,\dots, a_m)$

For each $x\in \mathbb{Z}_{\geq 0}$, $\sum_{\substack{a_i\leq x+k\\ \min(a_i)=x}} P_n(a_1,\dots, a_m)$ can be partitioned into the path-height functions with the same parameters up to rearrangement.  In this sum, there are $\binom{m}{m_0,\dots, m_k}$ path-height functions that are equal to $Q_n^{(x)}(m_0,\dots, m_k)$, and so 
\begin{align*}
G_n(m,k)=  \sum_{x = 0}^\infty \sum_{\substack{a_i\leq x+k\\ \min(a_i)=x}} P_n(a_1,\dots, a_m) =&\ \sum_{x = 0}^\infty \sum_{\substack{m_0>0\\ \sum m_i= m }} \binom{m}{m_0,\dots, m_k} Q^{(x)}_n(m_0,\dots, m_k)\\
=&\  \sum_{\substack{m_0>0\\ \sum m_i= m }}\binom{m}{m_0,\dots, m_k} Q_n(m_0,\dots, m_k) \qedhere
\end{align*}
\end{proof}

\begin{corollary}
\label{GShorten}
For $n\geq 1$, 
\begin{align*}
4G_n(m,1) - G_{n+1}(m,1) =&\  G_n(m+1,1)- 2(m-1)Q_n(m)\\
&\ + \sum_{k=1}^{m-1}(2\binom{m}{k}- \binom{m}{k+1} - \binom{m}{k-1})Q_n(k,m-k) 
\end{align*}
\end{corollary}
\begin{proof}
By Lemma \ref{GtoQ},
\begin{align*}
G_{n+1}(m,1)= \sum_{k=1}^m \binom{m}{k} Q_{n+1}(k,m-k)
\end{align*}
Letting $a_k := \binom{m}{k'}$, where $0\leq k' < m$ and $k\equiv k'$ (mod $m$), Lemma \ref{QShortenReform} turns the above into
\begin{align*}
G_{n+1}(m,1)&\ =\sum_{k=1}^{m} (2a_k+a_{k-1}+a_{k+1}) Q_{n}(k,m-k)\\
&\ - \sum_{k=1}^m (a_{k} + a_{k-1})Q_{n}(k,m+1-k)- a_m Q_{n}(m+1)
\end{align*}
Pulling out $Q_n(m)$ and plugging in for $a_k$ results in
\begin{align*}
G_{n+1}(m,1)=&\ (2m+2)Q_n(m) + \sum_{k=1}^{m-1} (2\binom{m}{k}+\binom{m}{k-1}+\binom{m}{k+1}) Q_{n}(k,m-k)
\end{align*}
\begin{align*}
&\ - \sum_{k=1}^m (\binom{m}{k} + \binom{m}{k-1})Q_{n}(k,m+1-k)- Q_{n}(m+1)\\
=&\ (2m+2)Q_n(m) + \sum_{k=1}^{m-1} (2\binom{m}{k}+\binom{m}{k-1}+\binom{m}{k+1}) Q_{n}(k,m-k)\\
&\ - \sum_{k=1}^{m+1} \binom{m+1}{k}Q_{n}(k,m+1-k)
\end{align*}
Which by Lemma \ref{GtoQ} means that
\begin{align*}
G_{n+1}(m,1)=&\ (2m+2)Q_n(m) + \sum_{k=1}^{m-1} (2\binom{m}{k}+\binom{m}{k-1}+\binom{m}{k+1}) Q_{n}(k,m-k)\\
&\ - G_n(m+1,1)
\end{align*}
As a result, 
\begin{align*}
4G_n(m,1) - G_{n+1}(m,1)= &\ 4\sum_{k=1}^m\binom{m}{k} Q_n(k,m-k)-(2m+2)Q_n(m)\\
&\ - \sum_{k=1}^{m-1} (2\binom{m}{k}+\binom{m}{k-1}+\binom{m}{k+1}) Q_{n}(k,m-k)+ G_n(m+1,1)\\
=&\ G_n(m+1,1)-(2m-2)Q_n(m)\\
&\ + \sum_{k=1}^{m-1} (2\binom{m}{k}-\binom{m}{k-1}-\binom{m}{k+1}) Q_{n}(k,m-k)\qedhere
\end{align*}\end{proof}
\section{Proof of Theorem \ref{m3Theorem} and \ref{m4Theorem}}

\noindent \emph{Proof of Theorem \ref{m3Theorem}:}
\begin{align*}
T(3,n) =&\  4T(2,n)-T(2,n+1)\\
=&\  4G_n(2,1)-G_{n+1}(2,1)\\
=&\  G_n(3,1) + (\sum_{k=1}^1 (2\binom{2}{k} - \binom{2}{k-1} - \binom{2}{k+1})Q_n(k,2-k)) + (2-2*2)Q_n(2)\\
=&\ G_n(3,1) + 2Q_n(1,1) - 2Q_n(2)\\
=&\ G_n(3,1) + 2(\sum_{k=1}^2(-1)^{k+1} Q_n(k,2-k))\\
=&\ G_n(3,1) + 2G_{n-1}(3,0)
\end{align*}	
	where the last step is by Corollary \ref{QAlternator}.\qed
	
	\noindent \emph{Proof of Theorem \ref{m4Theorem}:}
	\begin{align*}
	T(4,n) =&\  4T(3,n)-T(3,n+1)\\
	\end{align*}
	\begin{align*}
	=&\  4(G_n(3,1)+2G_{n-1}(3,0))-(G_{n+1}(3,1)+2G_{n}(3,0))\\
	=&\  (4G_{n}(3,1)-G_{n+1}(3,1)) + 8 G_{n-1}(3,0)-2G_{n}(3,0)\\
	=&\  (G_n(4,1) + (\sum_{k=1}^2 (2\binom{3}{k} - \binom{3}{k-1} - \binom{3}{k+1})Q_n(k,3-k))\\
	&\ + (2-3*2)Q_n(3)) + 8 Q_{n-1}(3)-2Q_{n}(3)\\
	=&\  G_n(4,1)+ 8 Q_{n-1}(3) + 2Q_n(1,2) + 2Q_n(2,1) -6Q_{n}(3)
	\end{align*}	 
	
	Applying Corollary \ref{QShortenab} to the third and fourth terms and Corollary \ref{QShortencor2} to the last term results in\\
	\begin{align*}
	T(4,n)=&\  G_n(4,1)+ 8 Q_{n-1}(3)\\
	&\ + 2(2Q_{n-1}(1,2)+Q_{n-1}(2,1)+Q_{n-1}(3)-Q_{n-1}(2,2)-Q_{n-1}(1,3))\\
	&\ + 2(2Q_{n-1}(2,1)+Q_{n-1}(3)+Q_{n-1}(1,2)-Q_{n-1}(3,1)-Q_{n-1}(2,2))\\
	&\ -6(2Q_{n-1}(3)+Q_{n-1}(1,2)+Q_{n-1}(2,1)-Q_{n-1}(1,3)-Q_{n-1}(4)-Q_{n-1}(3,1))\\
	=&\  G_n(4,1)-4Q_{n-1}(2,2)+4Q_{n-1}(1,3)+6Q_{n-1}(4)+4Q_{n-1}(3,1)\\
	=&\  G_n(4,1)+10Q_{n-1}(4)+ 4(Q_{n-1}(1,3)-Q_{n-1}(2,2)+Q_{n-1}(3,1)-Q_{n-1}(4))\\
	=&\  G_n(4,1) + 10G_{n-1}(4,0)+ 4G_{n-2}(5,0).
	\end{align*}
	
	Where the last step is by Corollary \ref{QAlternator}.\qed
	
	The proofs for the above theorems extend to the following for $m=5$:
\begin{equation}
	\label{m=5}
	T(5,n) = G_n(5,1) + 37G_{n-1}(5,0)+ 35G_{n-2}(6,0)+ 10G_{n-3}(7,0)-14G_{n-2}(5,0).
\end{equation}	

Notice that, due to the last term, the above equation is not in the form of Equation \ref{conjecture}, and is not even a positively weighted combinatorial interpretation.

\section*{Acknowledgment}
We thank the Summer Undergraduate Research Fellowship and the NASA Pennsylvania Space Grant Consortium for their generous support.

\end{document}